\documentclass[12pt,a4paper]{article} 

\usepackage{amsthm}
\usepackage{amsfonts}
\usepackage{amsmath}
\usepackage{amssymb}
\usepackage{hyperref}
\usepackage[utf8]{inputenc}
\usepackage{breqn}
\usepackage{cleveref}
\crefname{equation}{formula}{formulas}
\usepackage{bm}
\usepackage[affil-it]{authblk}
\usepackage{etoolbox}
\patchcmd{\thebibliography}{\section*}{\section}{}{}
\usepackage{nicefrac}

\newtheorem{thm}{Theorem}[section]
\newtheorem{lma}[thm]{Lemma}

\newtheoremstyle{break}
  {\topsep}
  {\topsep}
  {\itshape}
  {}
  {\bfseries}
  {.}
  {\newline}
  {}
\theoremstyle{break}

\newtheorem*{lmaB*}{Lemma}

\theoremstyle{definition}
\newtheorem{dfn}[thm]{Definition}
\newtheorem{rmrk}[thm]{Remark}

\newtheorem{fact}[thm]{Fact}

\newtheoremstyle{breakDefinition}
  {\topsep}
  {\topsep}
  {}
  {}
  {\bfseries}
  {.}
  {\newline}
  {}
\theoremstyle{breakDefinition}

\newcommand{\proofAVName}{Claim}

\DeclareMathOperator{\cMe}{c}

\DeclareMathOperator{\theory}{Th}

\DeclareMathOperator{\age}{Age}

\DeclareMathOperator{\diag}{Diag}

\newcommand{\fraisse}{Fra{\"i}ss{\'e}\hspace{4pt}}

\newcommand*{\card}[1][M]{\vert #1\vert}
\newcommand{\subsetfinite}{\subseteq_{\omega}}
\mathchardef\mhyphen="2D

\newcommand*{\frmDblI}[5][\psi]{
	\ifthenelse{\equal{#4}{}}
		{#1_{#5}(\overline{#2},#3)}
		{#1_{#5}(\overline{#2},#3\overline{#4})}
}

\newcommand*{\frmIExTup}[5][\psi]{
	\ifthenelse{\equal{#4}{}}
		{\exists #3#1_{#5}(\overline{#2},#3)}
		{\exists #3\overline{#4}#1_{#5}(\overline{#2},#3\overline{#4})}
}

\newcommand{\forkindepPrv}[1]{
  \mathrel{
    \mathop{
      \vcenter{
        \hbox{\oalign{\noalign{\kern-.3ex}\hfil$\vert$\hfil\cr
              \noalign{\kern-.7ex}
              $\smile$\cr\noalign{\kern-.3ex}}}
      }
    }\displaylimits_{#1}
  }
}

\newcommand{\indepPrv}[2][\Gamma]{
  \mathrel{
    \mathop{
      \vcenter{
        \hbox{{\oalign{\noalign{\kern-.3ex}\hfil$\hspace*{4pt}\vert^{#1} $\hfil\cr
              \noalign{\kern-.7ex}
              $ \smile $\cr\noalign{\kern-.3ex}}}}
      }
    }\displaylimits_{#2}
  }
}

\DeclareMathOperator{\cl}{cl}

\newcommand*{\freejoin}[3]{#1\sqcup_{#2}#3}

\newcommand{\minpair}[2][A]{#1\not\leq^{*}_{\min}#2}

\newcommand{\kpluszero}{\mathcal{K}^{+}_{0}}
\newcommand{\kpluszerobar}{\overline{\mathcal{K}}^{+}_{0}}

\DeclareMathOperator*{\sep}{sep}
\DeclareMathOperator{\uo}{uo}
\DeclareMathOperator{\univ}{UNIV}

\renewcommand{\proofAVName}{Calim}

\begin{document}

\title{Pseudofiniteness in Hrushovski Constructions}

\author{Ali Valizadeh
  \thanks{\texttt{valizadeh.ali@aut.ac.ir,}}
		}
\affil{	\small Amirkabir University of Technology, Iran}

\author{Massoud Pourmahdian
		\thanks{\texttt{pourmahd@ipm.ir}}
		}
\affil{ \small IPM(Institute for Research in Fundamental Sciences), Iran}

\date{}

\maketitle

\begin{abstract}
In a relational language consisting of a single relation $ R, $ we investigate pseudofiniteness of certain Hrushovski constructions obtained via predimension functions. It is notable that the arity of the relation $ R $ plays a crucial role in this context.  

When $ R $ is ternary, by extending the methods developed in \cite{Brody&Laskowski-OnRationalLimits}, we interpret $ \langle\mathbb{Q}^{+},<\rangle $ in the $ \langle\kpluszero,\leq^{*}\rangle $-generic and prove that this structure  is not pseudofinite. This provides a negative answer to the question posed in \cite{Evans&Wong-SomeRemarksonGen} (Question 2.6). This result, in fact, unfolds another aspect of complexity of this structure, along with undecidability and strict order property proved in \cite{Evans&Wong-SomeRemarksonGen} and \cite{Brody&Laskowski-OnRationalLimits}. On the other hand, when $ R $ is binary, it can be shown that the $ \langle\kpluszero,\leq^{*}\rangle $-generic is decidable and pseudofinite. 

\end{abstract}

\section{Introduction and Setup}\label{secIntro}
A complete $ \mathcal{L} $-theory $ T $ is called pseudofinite if for each $ \varphi\in T $ there exists a finite $ \mathcal{L} $-structure satisfying $ \varphi. $ A structure $ M $ is called pseudofinite if $ \theory(M) $ is pseudofinite. 
Since the introduction of the \fraisse-Hrushovski constructions, the question of whether the theories raising from these limits are pseudofinite has been a subject of interest.

Working in a relational language $ \mathcal{L} $ consisting of a single relation $ R $ of arity at least 2, it is known for an irrational $ \alpha\in(0,1), $ that the theory of the non-collapsed ab initio $ \langle\mathcal{K}_{\alpha},\leq_{\alpha}\rangle $-generic, $ \theory(M_{\alpha}), $ is the same as the almost sure theory of the class of random finite graphs with edge probability $ n^{-\alpha} $  (\cite{Baldwin&Shelah-Randomness}). In particular, this implies that the theory $ \theory(M_{\alpha}) $ is pseudofinite. Also, for a rational $ \alpha\in(0,1] $, it can be seen from the results in \cite{Brody&Laskowski-OnRationalLimits} that for an ascending sequence of irrationals $ \{\alpha_{i}\}_{i\in\omega} $ converging to $ \alpha, $ the $ \langle\mathcal{K}_{\alpha},\leq_{\alpha}\rangle $-generic $ M_{\alpha}, $ is elementary equivalent to an ultraproduct of $ M_{\alpha_{i}} $s. Hence, $ M_{\alpha} $ is again a pseudofinite structure.

On the other hand, in view of the fact that any strongly minimal pseudofinite structure is locally modular (\cite{Pillay-StroglyminimalPseudo}), Hrushovski's strongly minimal generic structure fails to be pseudofinite. 

Still, by considering an $ \alpha\in(0,1] $ and using the predimension function $ \delta_{\alpha}, $ one can equip the class of finite $ \mathcal{L} $-structures with another notion of \textit{closedness}, denoted by $ \leq^{*}_{\alpha}. $ This notion of closedness often leads to generic structures with unstable theories; one can refer to \cite{Pourmahd-SimpleGen, Pourmahd-SmoothClasses,Evans&Wong-SomeRemarksonGen} and \cite{Brody&Laskowski-OnRationalLimits} to investigate known results and further details.  

For a rational $ \alpha<1, $ using the methods developed in \cite{Evans&Wong-SomeRemarksonGen} and \cite{Brody&Laskowski-OnRationalLimits}, it can be shown that (a subtheory of) the theory of the $ \langle\mathcal{K}^{+}_{\alpha},\leq^{*}_{\alpha}\rangle $-generic interprets both finite graphs and Robinson arithmetic, hence has the strict order property and is undecidable. For $ \alpha=1, $ the same facts hold if $ R $ is a ternary relation. 
In this paper, by taking $ \alpha=1, $ we focus on the pseudofiniteness issue in the $ \langle\mathcal{K}^{+}_{\alpha},\leq^{*}_{\alpha}\rangle $-generic structure.

To be more precise, let $ \mathcal{L} $ be a relational language with only one relation $ R $ of arity $ n_{R}\in\{2,3\} $ and take $ \alpha=1. $ For any finite $ \mathcal{L} $-structure $ A $ in which $ R^{A} $ is symmetric and anti-reflexive let
\[ \delta(A):= |A| - |R[A]|, \]

where $ |R[A]| $ is the number of hyperedges, i.e.
\[ R[A]=\Big\{\{a_{1},\ldots,a_{n_{R}}\}\Big|(a_{1},\ldots,a_{n_{R}})\in R^{A}\Big\}. \]

Let $ \kpluszero $ be the following class of finite $ \mathcal{L} $-structures,
\[ \kpluszero:=\Big\{A \hspace*{3pt}\Big|\hspace*{5pt} |A|<\aleph_{0}, \forall B\subseteq A,\hspace*{3pt} \delta(B)>0 \Big\}. \]
As a convention, we assume that $ \kpluszero $ contains the empty set. We denote by $ \kpluszerobar $ the class of all $ \mathcal{L} $-structures $ M $ whose finite substructures lie in $ \kpluszero, $ namely $ \age(M)\subseteq\kpluszero. $

\begin{dfn}\label{dfnClosedness} Suppose that $ A,B\in\kpluszero. $
	\begin{itemize}
		\item[(i)] We say that $ A $ is \textit{closed} or \textit{strong} in $ B $ and in notations we write $ A\leq^{*}{B}, $ if  $ A\subseteq B $ and for any $ C\subseteq B $ with $ A\subsetneq C $ we have $ \delta(C)>\delta(A). $
		
		\item[(ii)] For $ M\in\kpluszerobar $ and a finite $ A\subseteq M $ we say that $ A $ is closed in $ M, $ denoted by $ A\leq^{*}M, $ if for any finite $ B\subseteq M $ with $ A\subseteq B $ we have that $ A\leq^{*}B. $
	\end{itemize}
\end{dfn}

In \Cref{secMain}, we further refine the techniques developed in \cite{Brody&Laskowski-OnRationalLimits} and show that the structure $ \langle\mathbb{Q}^{+},<\rangle $ is interpretable in the $ \langle\kpluszero,\leq^{*}\rangle $-generic structure $ \mathfrak{M} $ (\Cref{dfnGenericModel}). Hence, $ \theory(\mathfrak{M}) $ is not pseudofinite. This gives a negative answer to a question posed in \cite{Evans&Wong-SomeRemarksonGen} (Question 2.6). 
On the other hand, when $ R $ is a binary relation, it can be shown that $ \theory(\mathfrak{M}) $ is decidable and pseudofinite.

We recall some of the basic definitions in the context of \fraisse-Hrushovski constructions. The reader can refer to \cite{Wagner-Relational, Kueker&Laskowski-GenericStructures}, and \cite{Baldwin&Shi-StableGen} for more details.

\paragraph*{\textbf{Notation.}} $ \mathcal{L} $ is a relational language with only one ternary relation $ R. $ Finite $ \mathcal{L} $-structures are denoted by $ A,B,C,\ldots.  $ By $ M,N,\ldots $ we mean arbitrary $ \mathcal{L} $-structures. By $ A\subsetfinite M $ we mean that $ A $ is a finite substructure of $ M. $ Finally for $ A,B\subseteq C, $ the structure induced by $ C $ on $ A\cup B $ is denoted by $ AB. $

\begin{dfn}\label{dfnFreeJoin}
Suppose that $ N_{0}, N_{1}, N_{2}\in\kpluszerobar $ with $ N_{1}\cap N_{2} = N_{0}. $
The structure $ N $ is called the \textit{free join} or \textit{free amalgam of} $ N_{1} $ \textit{and} $ N_{2} $ \textit{over} $ N_{0} $, denoted by $ \freejoin{N_{1}}{N_{0}}{N_{2}}, $ if the universe of $ N $ is $ N_{1}\cup N_{2} $ and the following holds 
\[ R^{N} = R^{N_{1}}\cup R^{N_{2}}. \]
\end{dfn}
\begin{fact}\label{factFullAmalgam}
The class $ \langle\kpluszero,\leq^{*}\rangle $ has the \textit{full amalgamation property}, i.e., if $ N_{0}, N_{1}, N_{2}\in\kpluszerobar $ with $ N_{1}\cap N_{2} = N_{0}, N_{0}\leq^{*}N_{1} $ and $ N=\freejoin{N_{1}}{N_{0}}{N_{2}}, $ then $  N\in\kpluszerobar $ and $ N_{2}\leq^{*}N. $
\end{fact}

Recall that, since $ \langle\kpluszero,\leq^{*}\rangle $ has the amalgamation property, there is a unique countable model $ \mathfrak{M}=\langle M,R^{M}\rangle\in\kpluszerobar $ with the following properties
\begin{itemize}
\item[(i)] For every $ A\in\kpluszero, $ there is a closed embedding of $ A $ into $ \mathfrak{M}. $ (\textit{Universality})
\item[(ii)] For every $ A\leq^{*}B\in\kpluszero $ with $ A\leq^{*}\mathfrak{M}, $ there is a closed embedding of $ B $ over $ A $ in $ \mathfrak{M}. $ (\textit{*-homogeneity})
\item[(iii)] $ \mathfrak{M} $ is the union of a chain of finite structures $ \{A_{i}:i\in\omega\}, $ where for each $ i\in\omega $ we have that $ A_{i}\in\kpluszero $ and $ A_{i}\leq^{*} A_{i+1}. $  
\end{itemize}
\begin{dfn}\label{dfnGenericModel}
The model $ \mathfrak{M}, $ described above, is called the $ \langle\kpluszero,\leq^{*}\rangle $\textit{-generic}, or simply the generic model when the context is clear.
\end{dfn}

\begin{dfn}\label{dfnClosure}
For any $ N\in\kpluszerobar $ and $ A\subsetfinite N, $ the closure of $ A $ in $ N, $ denoted by $ \cl^{*}_{N}(A), $ is the unique minimal substructure of $ N $ that contains $ A $ and is closed in $ N. $ 
\end{dfn}

\begin{dfn}\label{dfnMinimal} For $ A,B\in\kpluszero, $
\begin{itemize}
\item[(i)] $ (A,B) $ is called a \textit{minimal pair}, denoted by $ \minpair{B}, $ if  $ A\subseteq B $ and $ A $ is closed in any proper substructure of $ B $ containing $ A, $ but not in $ B. $ We also say that $ B $ is a \textit{minimal extension} of $ A. $ If $ \delta(B/A)=0, $ we call $ (A,B) $ a $ 0 $-minimal pair. 
\item[(ii)] A minimal pair $ (A,B) $ is a \textit{biminimal pair} if every element of $ A $ is participating in a relation with at least one component in $ B\backslash A. $
\end{itemize}
\end{dfn}
\begin{fact}\label{factMinimalPairsDelta}
Suppose that $ A,B\in\kpluszero. $ If $ \minpair[A]{B}, $ then $ \delta(B/A)\leq 0 $ and for any $ C $ with $ A\subsetneq C\subsetneq B $ we have that $ \delta(C/A)> 0. $
\end{fact}

\begin{rmrk}
It is a well-known fact that $ \cl^{*}_{N}(A) $ contains all finite towers of minimal extensions over $ A $ in $ N. $ Moreover, as a consequence of property (ii) of a generic model, while the closure of each finite substructure of $ \mathfrak{M} $ is finite, this is not the case in an arbitrary model of $ \theory(\mathfrak{M}). $
\end{rmrk}

\paragraph*{\textbf{Notation.}} Given a minimal pair $ (A,B) $ in $ N\in\kpluszerobar, $ by a \textit{copy} of $ B $ over $ A, $ we mean the image of an embedding of $ B $ over $ A $ into $ N. $ Also, by $ \chi_{N}(B/A), $ we denote the number of disjoint copies of $ B $ over $ A $ in $ N. $ Unlike an arbitrary model of $ \theory(\mathfrak{M}), $  the value of $ \chi_{\mathfrak{M}}(B/A) $ is always finite.

%

\section{Main Results}\label{secMain}
\subsection{Case of a Ternary Relation}
We assume that $ R $ is a ternary relation and will work inside the $ \langle\kpluszero,\leq^{*}\rangle $-generic structure $ \mathfrak{M}. $ Hence, all the closures are taken within $ \mathfrak{M}. $ To ease the notations, we drop the subscript $ \mathfrak{M} $ from $ \cl^{*}_{\mathfrak{M}}(A) $ and $ \chi_{\mathfrak{M}}(B/A), $ and will write $ \cl^{*}(A) $ and $ \chi(B/A) $ instead. We follow the terminology used in \cite{Brody&Laskowski-OnRationalLimits}. In particular, for a set $ S $ and a natural number $ k\geq1, $ we denote by $ [S]^{k} $ the set of all subsets $ Y\subseteq S $ with $ \card[Y]=k. $ 

Based on the following lemmas, given a natural number $ k\geq 1, $ there exists a definable relation $ R^{k}(x_{1}, \ldots,x_{k};v) $ such that for a fixed finite subset $ S\subsetfinite \mathfrak{M}, $ and every set $ X\subset[S]^{k}, $ there exists a ``\textit{code}'' $ v\in \mathfrak{M}, $ depending on $ S $ and $ X, $ that $ X=R^{k}(\mathfrak{M},v)\cap[S]^{k}. $

We recall Lemma 3.2 from \cite{Brody&Laskowski-OnRationalLimits} that is crucial for \Cref{lmaDefinability}. It worth noting that this lemma is obtained more easily in the present context. 
Once having the following lemma for ternary relations, \Cref{lmaDefinability} can be obtained using the same proof as given in Proposition 3.3 of \cite{Brody&Laskowski-OnRationalLimits}.
\begin{lma}\label{lmaMinPairExists}
For any natural number $ n\geq 1, $ there exists a natural number $ m $ such that for every $ A\in\kpluszero $ of size $ n $ there is a structure $ C\in\kpluszero $ with $ \card[C\backslash A]=m $ that is a $ 0 $-biminimal extension of $ A. $
\end{lma}
\begin{proof}
For $ n=1, $ take $ m=3. $ Now for a structure $ A=\{a\}, $ let $ C=\{a,c_{1},c_{2},c_{3}\} $ with $  R^{C} $ being the symmetric closure of the set $\{(a,c_{i},c_{j}) | \text{ for all } i\neq j\}. $

For $ n\geq2, $ let $ m=n. $ Now suppose $ A\in\kpluszero $ is a fixed structure with universe $ \{a_{1},\ldots,a_{n}\}. $ Let $ C=A\cup\{c_{1},\ldots,c_{n}\} $ and $  R^{C} $ be the symmetric closure of the following set 
\begin{align*}
 R^{A}\cup\Big\{(a_{i},c_{i},c_{i+1})\Big| \text{ for all } 1\leq i\leq n-1\Big\}\cup
 \Big\{(a_{n},c_{1},c_{n})\Big\}.
\end{align*}
One can observe that the introduced structure $ C $ is a $ 0 $-biminimal extension of $ A. $ 

Note that, the above construction depends only on the cardinality of $ A. $ Therefore, for any $ A'\in\kpluszero $ with $ \card[A']=n, $ one can find a $ 0 $-biminimal extension $ C'\in\kpluszero $ of $ A' $ with $ \card[C'\backslash A']=m. $
\end{proof}

\begin{lma}\label{lmaDefinability}
For any $ k\in\omega $ there is a definable relation $ R^{k}(x_{1},\ldots,x_{k};y) $, symmetric in the first $ k $ variables, such that for any $ S\subsetfinite \mathfrak{M}$ and any $ X\subseteq \left[S\right]^{k} $ there exists some $ v\in \mathfrak{M} $ such that for any $ a_{1},\ldots,a_{k}\in \mathfrak{M} $ we have
\[ \mathfrak{M}\models R^{k}(a_{1},\ldots,a_{k};v)\quad\quad\iff\quad\quad \{a_{1},\ldots,a_{k}\}\in X. \]\qed
\end{lma}
\paragraph*{\textbf{Notation.}} For a fixed $ v, $ we denote the formula $ R^{k}(\bar{x};v) $ by $ R^{k}_{v}(\bar{x}). $

\paragraph*{\textbf{Setup.}}
Fix three finite structures $ A,B,C\in\kpluszero $ with $ B $ and $ C $ being two non-isomorphic $ 0 $-minimal extensions of $ A. $ Let $ \mathcal{A} $ be the set of all tuples $ \bar{a}\in \mathfrak{M} $ satisfying the following conditions.
\begin{itemize}
\item[(i)] $ \bar{a}\cong A. $
\item[(ii)] There exists at least one copy of $  B $ and at least one copy of $ C $ over $ \bar{a}, $ namely $ \chi(B/\bar{a})\geq1 $ and $ \chi(C/\bar{a})\geq1. $
\item[(iii)] Any two distinct copies of $ B $ (respectively copies of $ C $) over $ \bar{a} $ are disjoint.
\item[(v)] No copy of $ B $ intersects a copy of $ C $ over $ \bar{a}. $
\end{itemize}

Note that, using full amalgamation (\Cref{factFullAmalgam}) and genericity of $ \mathfrak{M}, $ there are infinitely many copies of $ A $ in $ \mathfrak{M} $ satisfying the above conditions. Furthermore, for each $ p/q\in\mathbb{Q}^{+}, $ one can build a structure $ D $ that consists of $ A $ together with $ p $ many copies of $ B $ and $ q $ many copies of $ C, $ all being freely amalgamated over $ A. $ Using genericity of $ \mathfrak{M},$ there is a closed embedding of $ D $ into $ \mathfrak{M}. $

For a tuple $ \bar{a}\in\mathcal{A}, $ let $ \mathcal{B}_{\bar{a}} $ be the following set
\[ \mathcal{B}_{\bar{a}}:=\Big\{\bar{b}\in \mathfrak{M}\hspace*{5pt}\Big|\hspace*{5pt}\bar{b}\cap\bar{a}=\varnothing, \bar{a}\bar{b} \text{ is a copy of } B \text{ over } \bar{a}\Big\}. \]
Also, let $ \mathbf{B}_{\bar{a}} $ denote the union of $ \mathcal{B}_{\bar{a}}. $ We can define $ \mathcal{C}_{\bar{a}} $ and $ \mathbf{C}_{\bar{a}} $ in a similar way.

For a tuple $ \bar{a}\in\mathcal{A}, $ let a $ B $\textit{-basis} for $ \bar{a} $ be a subset of $ \mathbf{B}_{\bar{a}} $ that contains exactly one element from each copy of $ B $ over $ \bar{a}. $ Analogously, we can define a $ C $\textit{-basis} for $ \bar{a}. $

We can equip $ \mathcal{A} $ with an equivalence relation defined as $ \bar{a}\sim\bar{a}' $ if and only if 
\[ \chi(B/\bar{a}).\chi(C/\bar{a}')=\chi(B/\bar{a}').\chi(C/\bar{a}). \]

Since we are working in the generic model, the number of copies of $ B $ or $ C $ over each copy of $ A $ is always finite. Hence, it is easy to see that $ \sim $ defines an equivalence relation on $ \mathcal{A}. $

In fact, by genericity of $ \mathfrak{M}, $ each equivalence class $ [\bar{a}]_{\sim} $ corresponds to a unique non-negative rational number $ p/q $ where $ p=\chi(B/\bar{a}) $ and $ q=\chi(C/\bar{a}). $ Having this intuition in mind, one can naturally define an order $ \prec $ on $ \mathcal{A}/\sim $ in such a way that $ \langle\mathcal{A}/\sim, \prec\rangle\cong\langle\mathbb{Q}^{+},<\rangle. $ 

So, let $ [\bar{a}]_{\sim} \prec[\bar{a}']_{\sim} $ if and only if
\[ \chi(B/\bar{a}).\chi(C/\bar{a}')<\chi(B/\bar{a}').\chi(C/\bar{a}). \]

We can express the above relations in terms of corresponding bases. Namely, for $ \bar{a},\bar{a}'\in\mathcal{A}, $ we have that $ \bar{a}\sim\bar{a}' $ if and only if for some (equivalently for all) $ X_{\bar{a}}, X_{\bar{a}'}, Y_{\bar{a}} $ and $ Y_{\bar{a}'} $ that are respectively $ B $-bases and $ C $-bases for $ \bar{a} $ and $ \bar{a}', $ we have that
\[ \card[X_{\bar{a}}]\card[Y_{\bar{a}'}]=\card[X_{\bar{a}'}]\card[Y_{\bar{a}}]. \]

A similar description for $ \prec $ can be given naturally in terms of bases.  

\Cref{thmQInterpretable} states that, actually the structure $ \langle\mathcal{A}/\sim,\prec\rangle $ is interpretable in $ \mathfrak{M}. $ The proof of this theorem is based on a proper implementation of \Cref{lmaDefinability} to ``code'' the set theoretic notions such as Cartesian product, bijection and injection in definable families of finite sets. 

For any two disjoint non-empty sets $ A $ and $ B, $ let $ \mathbb{D}_{A\times B}\subseteq [A\cup B]^{2} $ be the following set
\[ \Big\{\{x,y\}\Big|x\in A, y\in B\Big\}. \]
Note that $ \mathbb{D}_{A\times B} $ has the same cardinality as $ A\times B. $ We call $  \mathbb{D}_{A\times B} $ the \textit{unordered Cartesian product} of $ A $ and $ B; $ we denote it by $ A\times_{\uo}B. $

In order to describe a bijection between arbitrary sets $ A $ and $ B, $ using some basic set theoretic techniques, one can describe a set $ D $ of unordered pairs, namely $ D\subseteq [A\Delta B]^{2}, $ whose existence gives rise to a bijection between $ A\backslash B $ and $ B\backslash A $ (and consequently from $ A $ to $ B $). Likewise, one can describe a set $ D'\subseteq [A\Delta B]^{2} $ leading to existence of an injection from $ A $ to $ B. $ 

Now, to prove that $ \theory(\mathfrak{M}) $ is not pseudofinite, it suffices to interpret $ \langle\mathcal{A}/\sim,\prec\rangle $ in $ \mathfrak{M}. $ 

\begin{thm}\label{thmQInterpretable}
Suppose that $ \mathcal{L}=\{R\} $ where $ R $ is a ternary relation. Also, suppose that $ \mathfrak{M} $ is the generic structure of the class $ \langle\kpluszero,\leq^{*}\rangle $. Then, $ \langle\mathbb{Q}^{+},<\rangle $ is interpretable in $ \mathfrak{M}. $ Therefore, $ \theory(\mathfrak{M}) $ is not pseudofinite.
\end{thm}

\begin{proof}
To prove this theorem, we use \Cref{lmaDefinability} and the setup provided 
above.

It can be easily verified that the conditions (i)-(iv) in the definition of $ \mathcal{A} $ are $ \varnothing $-definable. The proof of Theorem 3.5 of \cite{Brody&Laskowski-OnRationalLimits} guarantees the existence of a formula $ \gamma(\bar{y},\bar{z}) $ expressing that $ \chi(B/\bar{y})=\chi(B/\bar{z}) $ and $ \chi(C/\bar{y})=\chi(C/\bar{z}). $ 

Let $ \delta_{\sep}(\bar{x},\bar{y},\bar{z}) $ be a formula expressing that $ \bar{x},\bar{y},\bar{z}\in\mathcal{A}, $ conjunction with $ \gamma(\bar{y},\bar{z}), $ conjunction with the condition that each copy of $ B $ or $ C $ over $ \bar{z} $ is disjoint from every copy of $ B $ and $ C $ over $ \bar{x}. $ Now let $ {E}(\bar{x},\bar{y}) $ be the following formula 
\begin{align*}
\exists w&v_{2}v_{1}u_{4}u_{3}u_{2}u_{1}\bar{z} \Big[ 
\delta_{\sep}(\bar{x},\bar{y},\bar{z})
\wedge \text{``}R^{1}_{u_{1}}(\mathfrak{M}) \text{ is a } B\text{-basis for } \bar{x}\text{''}\\
&\wedge  \text{``}R^{1}_{u_{2}}(\mathfrak{M}) \text{ is a } C\text{-basis for } \bar{x}\text{''}
\wedge \text{``}R^{1}_{u_{3}}(\mathfrak{M}) \text{ is a } B\text{-basis for } \bar{z}\text{''}\\
&\wedge \text{``}R^{1}_{u_{4}}(\mathfrak{M})\text{ is a } C\text{-basis for } \bar{z}\text{''}
\wedge  \text{``}R^{2}_{v_{1}}(\mathfrak{M})=R^{1}_{u_{1}}(\mathfrak{M})\times_{\uo}R^{1}_{u_{4}}(\mathfrak{M}) \text{''}\\
&\wedge  \text{``}R^{2}_{v_{2}}(\mathfrak{M})=R^{1}_{u_{2}}(\mathfrak{M})\times_{\uo}R^{1}_{u_{3}}(\mathfrak{M}) \text{''}\\
&\wedge  \text{``} R^{4}_{w}(\mathfrak{M})\text{ defines a bijection between } R^{2}_{v_{1}}(\mathfrak{M}) \text{ and } R^{2}_{v_{2}}(\mathfrak{M})\text{''}\Big].
\end{align*}

Also, let $ O(\bar{x},\bar{y}) $ be the following 
\begin{align*}
\exists w&v_{2}v_{1}u_{4}u_{3}u_{2}u_{1}\bar{z} \Big[ 
\delta_{\sep}(\bar{x},\bar{y},\bar{z})
\wedge \text{``}R^{1}_{u_{1}}(\mathfrak{M}) \text{ is a } B\text{-basis for } \bar{x}\text{''}\\
&\wedge  \text{``}R^{1}_{u_{2}}(\mathfrak{M}) \text{ is a } C\text{-basis for } \bar{x}\text{''}
\wedge \text{``}R^{1}_{u_{3}}(\mathfrak{M}) \text{ is a } B\text{-basis for } \bar{z}\text{''}\\
&\wedge \text{``}R^{1}_{u_{4}}(\mathfrak{M})\text{ is a } C\text{-basis for } \bar{z}\text{''}
\wedge  \text{``}R^{2}_{v_{1}}(\mathfrak{M})=R^{1}_{u_{1}}(\mathfrak{M})\times_{\uo}R^{1}_{u_{4}}(\mathfrak{M}) \text{''}\\
&\wedge  \text{``}R^{2}_{v_{2}}(\mathfrak{M})=R^{1}_{u_{2}}(\mathfrak{M})\times_{\uo}R^{1}_{u_{3}}(\mathfrak{M}) \text{''}\\
&\wedge  \text{``} R^{4}_{w}(\mathfrak{M})\text{ defines an injection but not a bijection from } R^{2}_{v_{1}}(\mathfrak{M}) \text{ to } R^{2}_{v_{2}}(\mathfrak{M})\text{''}\Big].
\end{align*}

We show that for each $ \bar{a},\bar{a}'\in\mathcal{A} $ we have $ \bar{a}\sim\bar{a}' $ if and only if $ \mathfrak{M}\models E(\bar{a},\bar{a}'). $ 

Suppose that $ \bar{a}\sim\bar{a}'. $ Using the genericity of $ \mathfrak{M}, $ there exists a tuple $ \bar{a}''\in\mathcal{A} $ satisfying $ \delta_{\sep}(\bar{a},\bar{a}',\bar{a}''). $ In particular, we have that $ \bar{a}\sim\bar{a}''. $ Consider the set $ S=\mathbf{B}_{\bar{a}}\cup\mathbf{C}_{\bar{a}}\cup\mathbf{B}_{\bar{a}''}\cup\mathbf{C}_{\bar{a}''}, $ and let $ X_{\bar{a}}, X_{\bar{a}''}, Y_{\bar{a}} $ and $ Y_{\bar{a}''} $ be some $ B $-bases and $ C $-bases respectively. By applying \Cref{lmaDefinability} on $ S, $ there exist $ u_{1},u_{2},u_{3},u_{4}\in \mathfrak{M} $ with
$ R^{1}_{u_{1}}(\mathfrak{M})=X_{\bar{a}}, R^{1}_{u_{2}}(\mathfrak{M})=Y_{\bar{a}},R^{1}_{u_{3}}(\mathfrak{M})=X_{\bar{a}''} $ and $ R^{1}_{u_{4}}(\mathfrak{M})=Y_{\bar{a}''}. $

Set $ v_{1},v_{2}\in \mathfrak{M} $ to code respectively $ R^{1}_{u_{1}}(\mathfrak{M})\times_{\uo}R^{1}_{u_{4}}(\mathfrak{M}) $ and $ R^{1}_{u_{2}}(\mathfrak{M})\times_{\uo}R^{1}_{u_{3}}(\mathfrak{M}) $ as subsets of $ [S]^{2}. $

Since $ \bar{a}\sim\bar{a}'', $ we have that 
\[ \card[R^{1}_{u_{1}}(\mathfrak{M})]\times\card[ R^{1}_{u_{4}}(\mathfrak{M})]=\card[R^{1}_{u_{2}}(\mathfrak{M})]\times \card[R^{1}_{u_{3}}(\mathfrak{M})]. \]

But, $ \card[R^{1}_{u_{1}}(\mathfrak{M})\times_{\uo}R^{1}_{u_{4}}(\mathfrak{M})]= \card[R^{1}_{u_{1}}(\mathfrak{M})]\times\card[ R^{1}_{u_{4}}(\mathfrak{M})] $ and $ \card[R^{1}_{u_{2}}(\mathfrak{M})\times_{\uo}R^{1}_{u_{3}}(\mathfrak{M})]= \card[R^{1}_{u_{2}}(\mathfrak{M})]\times\card[ R^{1}_{u_{3}}(\mathfrak{M})]. $ Hence, there is a bijection, as a subset of $ [S]^{4} $, between $ R^{1}_{u_{1}}(\mathfrak{M})\times_{\uo}R^{1}_{u_{4}}(\mathfrak{M}) $ and $ R^{1}_{u_{2}}(\mathfrak{M})\times_{\uo}R^{1}_{u_{3}}(\mathfrak{M}) $ which can be coded by an element $ w\in \mathfrak{M}. $

For the other direction, suppose $ \mathfrak{M}\models E(\bar{a},\bar{a}'). $ Since we are working in the generic, for the tuple $ \bar{a}\in \mathfrak{M} $ we have that $ \chi(B/\bar{a}) $ and $ \chi(C/\bar{a}) $ are finite; similarly for $ \bar{a}'. $ Hence, the variables $ u_{1},u_{2},u_{3},u_{4}, v_{1},v_{2} $ and $ w $ find their appropriate set theoretic meanings. Therefore, $ w $ codes an actual bijection that yields the equivalence of $ \bar{a} $ and $ \bar{a}'. $ 

A similar argument can be applied to show that $ [\bar{a}]_{\sim} \prec[\bar{a}']_{\sim} $ if and only if $ \mathfrak{M}\models O(\bar{a},\bar{a}'). $
\end{proof}

\subsection{Case of a Binary Relation}
We turn to the case that $ R $ is a binary relation, i.e. it defines a graph. In this case, the model theoretic properties of the generic structure drastically change. In fact, thanks to an old result from random graph theory, there exists a complete axiomatization for $ \theory(\mathfrak{M}) $ yielding decidability and pseudofiniteness of this structure.

The following lemmas show that the class $ \langle\kpluszero,\leq^{*}\rangle $ has a natural graph theoretic interpretation.

\begin{lma}\label{lmaGeneralProperties}
	Suppose that $ A $ is a finite $ \mathcal{L} $-structure.
	\begin{itemize}
		\item[(i)] $ A\in\kpluszero\Leftrightarrow $ the number of edges in $ A $ is strictly less than the number of vertices of $ A\Leftrightarrow A $ is an acyclic graph.
		\item[(ii)] If $ A\in\kpluszero $ has $ k $ many connected components, then $ \delta(A)=k. $
		\item[(iii)] $ A\leq^{*}B\in\kpluszero $ if and only if $ B\backslash A $ is not connected to $ A. $
	\end{itemize}
\end{lma}
\begin{proof}
	Obvious using the elementary techniques of finite graph theory.
\end{proof}

\begin{lma}\label{lmaMinPair}
	Suppose that $ \minpair{B}\subsetfinite N\in\kpluszerobar, $ then $ B\backslash A $ is a singleton. Hence, we have the following.
		\begin{itemize}
			\item[(i)] If $ \delta(B/A)=0, $ then $ B $ consists of a single element connected to $ A $ with only one edge.
			\item[(ii)] If $ \delta(B/A)<0, $ then $ B $ is a singleton with at least two relations to $ A. $ Moreover, the number of distinct copies of $ B $ over $ A $ is 1.
		\end{itemize}
\end{lma}
\begin{proof}
	
	If $ B\backslash A $ has more than one element, then for each $ b\in B\backslash A $ we must have $ A\leq^{*}Ab, $ i.e. there is no relation between $ b $ and $ A. $ Consequently, there can not be any relation between $ B\backslash A $ and $ A. $ This means that $ \delta(B/A)=\delta(B\backslash A)\geq1 $ contradicting the fact that $ \minpair[A]{B}. $ Now items (i) and (ii) are clear.
\end{proof}

The following lemma shows that the notion of closedness in $ \langle\kpluszero,\leq^{*}\rangle $ is first-order expressible. Note that, in general, this notion is type definable. 
\begin{lma}\label{lmaGenericityFOExpressible}
For each $ n\in\omega, $ there is a formula $ \gamma_{\cMe}^{n}(\bar{x}) $ with $ \card[\bar{x}]=n $ such that for every $ N\in\kpluszerobar $ and $ \bar{a}\in N $ we have the following
\begin{flalign*}
&N\models \gamma_{\cMe}^{n}(\bar{a}) \Leftrightarrow \bar{a}\leq^{*}N. 
\end{flalign*} 
\end{lma}
\begin{proof}
By \Cref{lmaMinPair}, every minimal pair over $ \bar{a} $ consists of a single point with at least one relation to $ \bar{a}. $ Hence, being closed in $ N $ is equivalent to non-existence of such a point. 
\end{proof}

\begin{dfn}\label{dfnTUniv}
For any $ A\in\kpluszero $ with $ \card[A]=n, $ let $ \theta_{A} $ be the following sentence
\[ \exists\bar{x}\Big(\diag_{A}(\bar{x})\wedge \gamma_{\cMe}^{n}(\bar{x})\Big). \]
Now, let $ \univ $ be the collection of the sentences asserting that the relation $ R $ defines an acyclic graph together with the set $ \{\theta_{A}\hspace*{3pt}|\hspace*{3pt}A\in\kpluszero\}.$
\end{dfn}
It is obvious that $ \mathfrak{M} $ is a model of $ \univ. $ In fact, we show that $ \univ $ gives a complete axiomatization for $ \theory(\mathfrak{M}). $ To this end, we recall the following fact from random graph theory (Theorem 3.3.2 in \cite{Spencer-StrangeLogic}).
\begin{fact}\label{factSpencer}
Let $ G_{1}, G_{2} $ both be acyclic graphs in which every finite tree occurs as a component an infinite number of times. Then $ G_{1} $ and $ G_{2} $ are elementarily equivalent.
\end{fact}
\begin{lma}\label{lmaTUnivisUltraHomog}
Every model $ N $ of $ \univ $ is *-homogeneous. 
\end{lma}
\begin{proof}
Suppose that $ A\leq^{*}N $ and $ A\leq^{*}B\in\kpluszero. $ Let $ n=\card[A] $ and $ B' $ be the structure that is obtained from $ n $ copies of $ B $ being mutually freely amalgamated over the empty set. Using the universality of $ N, $ there is a closed embedding of $ B' $ into $ N. $ Hence, there is at least one copy of $ B $ in $ B' $ that is disconnected from $ A. $
\end{proof}

\begin{thm}\label{thmTUnivComplete}
Suppose that $ \mathcal{L}=\{R\} $ where $ R $ is a binary relation. Also, suppose that $ \mathfrak{M} $ is the generic structure of the class $ \langle\kpluszero,\leq^{*}\rangle $. Then
\begin{itemize}
\item[(i)] $ \univ $ is complete, hence $ \theory(\mathfrak{M}) $ is decidable.
\item[(ii)] $ \mathfrak{M} $ is pseudofinite.
\end{itemize}
\end{thm}
\begin{proof}
(i) Using \Cref{lmaTUnivisUltraHomog}, it can be easily seen that $ N\models \univ $ if and only if every finite acyclic tree occurs as a component infinitely many times in $ N. $ Hence, using \Cref{factSpencer}, every two models of $ \univ $ are elementarily equivalent.

(ii) Let $ \{A_{i}\}_{i\in\omega} $ be an enumeration of all finite trees. For each $ i\in\omega, $ let $ B_{i} $ be the free amalgamation of $ A_{0},\ldots,A_{i} $ over the empty set. Now, given a non-principal ultrafilter $ \mathcal{U}, $ it can be seen that the $ \prod_{\mathcal{U}} B_{i} $ is a model of $ \univ. $
\end{proof}

\begin{rmrk}
Using the results on \textit{ultraflat} graphs (\cite{Herre-SuperstableGraphs}), one can see that the $ \theory(\mathfrak{M}) $ (in the binary case) is superstable. Furthermore, it can be easily seen that $ \theory(\mathfrak{M}) $ is not small, hence giving another example of a strictly superstable Hrushovski construction. The question of finding such a structure was first asked by Baldwin (Question 12 in \cite{Baldwin-Pathological}) and was answered by Ikeda and Kikyo in \cite{Ikeda&Kikyo-OnSuperStableGenerics}. In a separate paper, we have addressed  the stability theoretic issues for similar Hrushovski generic acyclic graphs (\cite{Vali&Pourmahd-SSS}).
\end{rmrk}
\section{Concluding Remarks}
For a rational $ \alpha\in(0,1] $ and a relation $ R $ with $ n_{R}\geq 2, $ one can define the predimension function $ \delta_{\alpha} $ as 
\[ \delta_{\alpha}(A):=\card[A]-\alpha|R[A]|, \]
for every finite $ \mathcal{L} $-structure $ A. $

\Cref{thmQInterpretable} states that, by taking $ n_{R}=3 $ and $ \alpha=1, $ the theory of the $ \langle\mathcal{K}^{+}_{1},\leq^{*}_{1}\rangle $-generic is not pseudofinite.
\Cref{lmaMinPairExists} and \Cref{lmaDefinability} are key steps in our argument in the proof of \Cref{thmQInterpretable}. These lemmas were originally proved for $ \alpha\in(0,1)\cap\mathbb{Q} $ and a binary relation $ R, $ but as it is also mentioned in \cite{Brody&Laskowski-OnRationalLimits}, they can be obtained for a relation $ R $ with $ n_{R}\geq 3. $ Hence, the same machinery proves \Cref{thmQInterpretable} for $ \alpha\in(0,1)\cap\mathbb{Q}. $ 

To sum up this paper with the available results in the literature (\cite{Brody&Laskowski-OnRationalLimits} and \cite{Evans&Wong-SomeRemarksonGen}), we establish the following theorem.

\begin{thm}\label{thmFinalTheorem}
Suppose that $ \mathcal{L} $ is a language consisting of a single relation $ R $ with arity  $ n_{R}.$ Also suppose that for $ \alpha\in(0,1], $ the $ \langle\mathcal{K}^{+}_{\alpha},\leq^{*}_{\alpha}\rangle $-generic structure is denoted by $ \mathfrak{M}_{\alpha}. $
\begin{itemize}
\item[(i)] If $ n_{R}\geq 3, $ then for any $ \alpha\in(0,1]\cap\mathbb{Q}, $ the theory of $ \mathfrak{M}_{\alpha} $ is not pseudofinite. 
\item[(ii)] If $ n_{R}=2, $ then for any $ \alpha\in(0,1)\cap\mathbb{Q}, $ the theory of $ \mathfrak{M}_{\alpha} $ is not pseudofinite.
\item[(iii)] If $ n_{R}=2, $ then for $ \alpha=1, $ the theory of $ \mathfrak{M}_{\alpha} $ is pseudofinite (\Cref{thmTUnivComplete}).
\end{itemize}
Moreover, on contrary to the cases (i) and (ii), the theory of $ \mathfrak{M}_{\alpha} $ in case (iii) is decidable.
\end{thm}


%
%

\paragraph*{\textbf{Acknowledgement.}} Some parts of this work were developed during our visit to Intitute Henri Poincar\'e (IHP). The authors would like to thank IHP and CIMPA for supporting our attendance in the trimester MOCOVA 2018 holding at IHP. We would also like to thank J. Baldwin, C. Laskowski, D. Macpherson, K. Tent, and the anonymous referee for their helpful comments and discussions.

\bibliographystyle{alpha}
\bibliography{../../01_AppFiles/refArticles,../../01_AppFiles/refBooks}

\end{document}